\documentclass[12pt, twoside]{article}
\usepackage{amsmath,amsthm,amssymb}
\usepackage{times}
\usepackage{enumerate}

\def\titlerunning#1{\gdef\titrun{#1}}
\makeatletter
\def\author#1{\gdef\autrun{\def\and{\unskip, }#1}\gdef\@author{#1}}
\def\address#1{{\def\and{\\\hspace*{18pt}}\renewcommand{\thefootnote}{}%
\footnote {#1}}%
\markboth{\autrun}{\titrun}}
\makeatother
\def\email#1{e-mail: #1}
\def\subjclass#1{{\renewcommand{\thefootnote}{}%
\footnote{\emph{Mathematics Subject Classification (2010):} #1}}}

\usepackage{paralist}
\newcommand\Sym{\mathfrak S}
\newcommand\CC{\mathcal C}
\usepackage{tikz}
\usetikzlibrary{matrix}

\usepackage{color}
\definecolor{grau}{rgb}{0.8,0.8,0.8}
\definecolor{hellgrau}{rgb}{0.95,0.95,0.95}
\usepackage[colorlinks=false, urlbordercolor=hellgrau, citebordercolor=grau, anchorcolor=grau, linkbordercolor=grau]{hyperref}

\theoremstyle{plain}
\newtheorem{theorem}{Theorem}[section]

\newtheorem{corollary}[theorem]{Corollary}
\newtheorem{lemma}[theorem]{Lemma}
\newtheorem{proposition}[theorem]{Proposition}
\theoremstyle{definition}

\newtheorem{remark}[theorem]{Remark} 
\newcommand\R{\mathbb{R}}
 
\begin{document}




\titlerunning{Optimal bounds for the colored Tverberg problem}

\title{Optimal bounds for the colored Tverberg problem%
\footnote{The results of this paper were announced by the third author at an IPAM workshop ``Combinatorial Geometry''
 on October 22, 2009. 
 The preprint version was released on 
 October 26, 2009, \href{http://arXiv.org/abs/0910.4987v1}{arXiv:0910.4987v1}.}}
 
\author{Pavle V. M. Blagojevi\'{c} 
\and 
Benjamin Matschke 
\and
G\"unter M. Ziegler}

\date{}

\maketitle

\address{P.V.M. Blagojevi\'{c}: Mathemati\v cki Institut SANU, Knez Mihailova 36, 11001 Beograd, Serbia, and 
Institut f\"ur Mathematik, FU Berlin, Arnimallee 2, 14195 Berlin, Germany; \email{pavleb@mi.sanu.ac.rs, blagojevic@math.fu-berlin.de}
\and
B. Matschke: Max Planck Institute for Mathematics, Vivatsgasse 7, 53111 Bonn, Germany; \email{matschke@mpim-bonn.mpg.de}
\and
G.M. Ziegler: Institut f\"ur Mathematik, FU Berlin, Arnimallee 2, 14195 Berlin, Germany; \email{ziegler@math.fu-berlin.de}
}


\maketitle

\begin{abstract}
\noindent
We prove a ``Tverberg type'' multiple intersection theorem.
It strengthens the prime case of the original Tverberg theorem
from 1966, as well as
the topological Tverberg theorem of B\'{a}r\'{a}ny et al.\ (1980),
 by adding color constraints.
It also provides an improved bound for the
(topological) colored Tverberg problem of B\'{a}r\'{a}ny \& Larman (1992)
that is tight in the prime case and asymptotically optimal in the
general case.
The proof is based on relative equivariant obstruction theory.
\end{abstract}

\section{Introduction}

Tverberg's theorem from 1966 \cite{Tverberg-1}  
claims that any family of $(d+1)(r-1)+1$ points in~$\R^d$ can
be partitioned into $r$ sets whose convex hulls
intersect; a look at the codimensions of intersections shows that
the number $(d+1)(r-1)+1$ of points is minimal for this.

In their 1990 study of halving lines and halving planes, B\'ar\'any,
F\"uredi \& Lov\'asz \cite{BFL} observed ``we need a colored version
of Tverberg's theorem'' and provided a first case, for three triangles
in the plane. In response to this, B\'{a}r\'{a}ny \& Larman
\cite{Bar-Lar} in 1992 formulated the following general problem and
proved it for the planar case.
\medskip

\noindent
\textbf{The colored Tverberg problem:}
\emph{Determine the smallest number $t=t(d,r)$ such that for every collection
$\CC =C_{0}\uplus \dots\uplus C_{d}$ of points in $\mathbb{R}^d$
with $|C_i|\ge t$, there are $r$
disjoint subcollections $F_{1},\dots,F_{r}$ of~$\CC $ satisfying
\begin{compactenum}[\rm (A)]
\item $|C_{i}\cap F_{j}|\le 1$ for every $i\in\{0,\dots,d\},\,j\in \{1,\dots,r\}$, and
\item $\mathrm{conv\,}(F_1)\cap\dots\cap\mathrm{conv\,}(F_r)\neq\emptyset$.
\end{compactenum}}
\medskip

A family of disjoint subcollections $F_{1},\dots,F_{r}$ of $\CC$ satisfying condition {\rm (A)}, i.e.,
that contain at most one point from each \emph{color class} $C_i$, is called a \emph{colored $r$-partition}.
(We do not require $F_1\cup\dots\cup F_r=\CC$ for this.)
We allow color classes to be multisets of points in~$\R^d$; in this case the cardinalities have to account for these.
This convention is compatible with the phrasing of the colored Tverberg problem and its topological generalization, where one replaces 
the collection of points $\CC$ in $\R^d$ by 
the images of the vertices of a $(|\CC|-1)$-simplex $\Delta_{|\CC|-1}$ under an (affine resp.\ continuous) map to~$\R^d$.

A colored $r$-partition $F_{1},\dots,F_{r}$ having in addition property {\rm (B)} is a \emph{colored Tverberg $r$-partition}.

A trivial lower bound is $t(d,r)\ge r$:
Collections $\CC$ of only $(r-1)(d+1)$
points in general position do not admit an intersecting
$r$-partition, again by codimension reasons.

B\'{a}r\'{a}ny and Larman showed that the trivial
lower bound is tight in the cases $t(1,r)=r $ and
$t(2,r)=r$, presented a proof by Lov\'{a}sz for $t(d,2)=2$, and
conjectured the following equality that is the main content of the colored Tverberg problem.
\medskip

\noindent
\textbf{The B\'{a}r\'{a}ny--Larman conjecture:}
\emph{$t(d,r)=r$ for all $r\ge2$ and $d\ge1$.}
\medskip

Still in 1992, \v{Z}ivaljevi\'{c} \& Vre\'{c}ica \cite{ZV-1}
established for $r$ prime the upper bound
$t(d,r)\le 2r-1$.   
The same bound holds for prime powers according to
\v{Z}ivaljevi\'{c} \cite{guide2}. The bound for primes also yields
bounds for arbitrary $r$: For example, one gets
$t(d,r)\le 4r-3$,  
since there is a prime $p$ (and certainly a prime power!) between $r$
and~$2r$.
 
As in the case of Tverberg's classical theorem,
one can consider a topological version of the colored Tverberg problem.
\medskip

\noindent
\textbf{The topological Tverberg theorem:}
(\cite{BaBB-81}, \cite[Sect.~6.4]{MatousekBZ:BU})
\emph{Let $r\ge2$ be a prime power, $d\ge1$, and $N=(d+1)(r-1)$.
Then for every continuous map $f$ of an $N$-simplex $\Delta_N$
to $\R^d$ there are  $r$ disjoint faces $F_{1},\dots,F_{r}$ of~$\Delta_N$
whose images under $f$ intersect in~$\R^d$.}
\medskip

\noindent
\textbf{The topological colored Tverberg problem:}
\emph{Determine the smallest number $t=tt(d,r)$ such that for every
simplex $\Delta$ with $(d+1)$-colored vertex set
$\CC =C_{0}\uplus \dots\uplus C_{d}$,  with $|C_{i}|\ge t$ for all~$i$,
and for every continuous map $f:\Delta\rightarrow\mathbb{R}^d$,
there are $r$ disjoint faces $F_{1},\dots,F_{r}$ of~$\Delta$ satisfying 
\begin{compactenum}[\rm (A)]
\item $|C_{i}\cap F_{j}|\le 1$ for every $i\in \{0,\dots,d\},\,j\in\{1,\dots,r\}$, and
\item $f(F_1)\cap\dots\cap f(F_r)\neq\emptyset$.
\end{compactenum}}
\medskip

A family of faces $F_{1},\dots,F_{r}$ satisfying both conditions {\rm (A)} and {\rm (B)} is called a \emph{topological colored Tverberg $r$-partition}.

The argument from \cite{ZV-1} and \cite{guide2} gives the same upper bound
$tt(d,r)\le 2r-1$ for $r$ a prime power, and consequently the upper
bound $tt(d,r)\le 4r-3$ for arbitrary $r$.
Notice that $t(d,r)\le tt(d,r)$. 
\medskip

\noindent
\textbf{The topological B\'{a}r\'{a}ny--Larman conjecture:}
\emph{$tt(d,r)=r$ for all $r\ge2$ and $d\ge1$.}
\medskip

The Lov\'{a}sz proof for $t(d,2)=2$ presented in \cite{Bar-Lar}
is topological and thus also valid for
the topological B\'{a}r\'{a}ny--Larman conjecture. Therefore $tt(d,2)=2$.

The general case of the topological B\'{a}r\'{a}ny--Larman conjecture
would classically be
approached via a study of the existence of an $\Sym_r$-equivariant map
\begin{equation}
\Delta_{r,|C_{0}| } * \dots * \Delta_{r,|C_{d}|}
\ \ \longrightarrow_{\Sym_r}\ \ S(W_{r}^{\oplus (d+1)})
\ = \ S^{(r-1)(d+1)-1},
\label{eq:Map-colored-Tverberg}
\end{equation}
where $W_{r}$ denotes the $(r-1)$-dimensional real representation of
$\Sym_{r}$ obtained by restricting the coordinate permutation action on
$\mathbb{R}^r$ to
$\{(\xi_{1},\dots,\xi_{r})\in\mathbb{R}^{r} : \xi_{1}+\dots+\xi_{r}=0\}$
and~$\Delta_{r,n}$ is the $r\times n$ chessboard complex~$([r])^{*n}_{\Delta(2)}$;
cf.~\cite[Remark after Theorem~6.8.2]{MatousekBZ:BU}.
However, we will establish in Proposition~\ref{prop:fails}
that this approach fails when applied to the
colored Tverberg problem directly, due to the fact that
the square chessboard complexes~$\Delta_{r,r}$ admit
$\Sym_r$-equivariant collapses that reduce the dimension.

In the following, we circumvent this problem by a different,
particular choice of parameters,
which produces chessboard complexes $\Delta_{r,r-1}$
that are closed pseudomanifolds and thus do not admit collapses.
 
\section{Statement of the main results}

Our main result is the following strengthening of
(the prime case of) the topological Tverberg theorem.

\begin{theorem}\label{thm:main2}    
Let $r\ge2$ be prime, $d\ge1$, and $N:=(r-1)(d+1)$. 
Let $\Delta_N$ be an $N$-dimensional simplex with a partition of its vertex set
into $m+1$ parts (``color classes'') 
\[
 \CC \ \ =\ \ C_{0}\uplus \dots\uplus C_m,
\]
with $|C_i|\le r-1$ for all $i$. 

Then for every continuous map $f:\Delta_N\rightarrow\mathbb{R}^d$, 
there is a colored $r$-partition, given by disjoint faces
$F_{1},\dots,F_{r}$ of~$\Delta_N$ whose
images under $f$ intersect, that is,
\begin{compactenum}[\rm (A)]
\item $|C_{i}\cap F_{j}|\le 1$ for every $i\in \{0,\dots,m\},\, j\in\{1,\dots,r\}$, and
\item $f(F_1)\cap\dots\cap f(F_r)  \neq\emptyset$.
\end{compactenum}
\end{theorem}

The requirement $|C_i|\le r-1$ forces there to be at least $d + 2$ non-empty color classes.
Theorem \ref{thm:main2} is tight in the sense that there would exist counter-examples $f$ if $|C_0|=r$ and $|C_1|=\cdots=|C_m| =1$.

Our first step will be to reduce Theorem~\ref{thm:main2} to the following
special case.

\begin{theorem}\label{thm:main}
    Let $r\ge2$ be prime, $d\ge1$, and $N:=(r-1)(d+1)$.
    Let $\Delta_N$ be an $N$-dimensional simplex with
    a partition of its vertex set into $d+2$ parts
    \[
    \CC \ \ =\ \ C_{0}\uplus \dots\uplus C_{d}\uplus C_{d+1},
    \]
    with $|C_i|=r-1$ for $i\le d$ and $|C_{d+1}|=1$.

    Then for
    every continuous map $f:\Delta_N\rightarrow\mathbb{R}^d$,
    there are $r$ disjoint faces $F_{1},\dots,F_{r}$ of~$\Delta_N$ satisfying
\begin{compactenum}[\rm (A)]
\item $|C_{i}\cap F_{j}|\le 1$ for every $i\in \{0,\dots,d+1\},\, j\in\{1,\dots,r\}$, and
\item $f(F_1)\cap\dots\cap f(F_r)  \neq\emptyset$.
\end{compactenum}
\end{theorem}

\begin{proof}[Reduction of Theorem~\ref{thm:main2} to Theorem~\ref{thm:main}.]
Let $f:\Delta_N\to\R^d$ be a continuous map  and $C_0 \uplus\dots\uplus C_m$ a coloring of the vertex set of~$\Delta_N$.
Let $N':=(r-1)(m+1)$ and $C_{m+1}:=\emptyset$.
We enlarge the color classes $C_i$ by adding $N'-N = (r-1)(m-d)$ new vertices and obtain new color classes $C'_0 , \dots ,C'_{m+1}$, such that $C_i\subseteq C'_i$ for all $i$, 
$|C'_0 | = \dots = |C'_m | = r - 1$ and $|C'_{m+1}| = 1$. 
Using the map $f$, we construct a
new map $f': \Delta_{N'} \rightarrow \R^{m}$, as follows:
We regard $\R^d$ as the subspace of $\R^{m}$ where the last $m-d$ coordinates  
are zero. So we let $f'$ be the same as $f$ on the $N$-dimensional front
face of $\Delta_{N'}$. We assemble the further $N' - N$ vertices into  $m - d$
groups $V_1,\dots, V_{m-d}$ of $r-1$ vertices each. The vertices in~$V_i$ shall be   
mapped to $e_{d+i}$, the $(d+i)$th standard basis vector of $\R^{m}$. We extend
this map using barycentric coordinates to all of $\Delta_{N'}$ in order to obtain $f'$.
We apply Theorem~\ref{thm:main} to $f'$ and the coloring $C'_0 ,\dots,C'_{m+1}$ and obtain
disjoint faces $F'_1 ,\dots, F'_r$ of~$\Delta_{N'}$. Let $F_i := F'_i \cap\Delta_N$ be the intersection of
$F'_i$ with the $N$-dimensional front face of $\Delta_{N'}$.
By construction of $f'$, the nonempty intersection $f'(F'_1 )\cap\dots\cap f'(F'_r)$ lies in $\R^d$.
Therefore, already $F_1,\dots, F_r$ is a topological colored Tverberg $r$-partition for $f'$, and hence it is also a topological colored Tverberg $r$-partition for $f$:
We have $f(F_1)\cap\dots\cap f(F_r)\neq\emptyset$.
\end{proof}
	
Such a reduction previously appeared in Sarkaria's proof for the prime power Tverberg theorem \cite[(2.7.3)]{Sarkaria-primepower};
see also de Longueville's exposition \cite[Prop.~2.5]{deL01}.
 
Either of our Theorems \ref{thm:main2} and \ref{thm:main} immediately implies the topological Tverberg theorem
for the case when $r$ is a prime, as any colored Tverberg partition, as provided by Theorems \ref{thm:main2} and \ref{thm:main}, is, in particular, a Tverberg partition
(if one ignores the color constraints).
Thus Theorems \ref{thm:main2} and \ref{thm:main} are ``constrained'' Tverberg theorem as recently discussed by Hell \cite{Hell-2}.

More importantly, however, Theorem~\ref{thm:main}
implies the topological B\'{a}r\'{a}ny--Larman
conjecture for the case when $r+1$ is a prime, as follows.

\begin{corollary}
\label{Th-Result3} If $r+1$ is prime, then $t(d,r)=tt(d,r)=r$.
\end{corollary}

\begin{proof}
    We prove that if $r\ge3$ is prime, then $tt(d,r-1)\le r-1$.
    For this, let $\Delta_{N-1}$ be a simplex where $N=(r-1)(d+1)$ and with vertex set $\CC =C_{0}\uplus \dots\uplus C_{d}$, $|C_i|=r-1$ for all $i$, and let
    $f:\Delta_{N-1}\rightarrow\R^d$ be continuous.
    Extend this to a map $\Delta_{N}\rightarrow\R^d$, where
    $\Delta_N$ has an extra vertex $v_N$, and set $C_{d+1}:=\{v_N\}$.
    Then Theorem~\ref{thm:main} can be applied, and yields a topological colored Tverberg $r$-partition.
    Ignore the part that contains~$v_N$.
\end{proof}

Using estimates on prime numbers one can derive from this
tight bounds for the colored Tverberg problem also in
the general case. The classical Bertrand's postulate
(``For every $r$ there is a prime $p$ with $r+1\le p<2r$'')
can be used here, but there are also much stronger
estimates available, such
as the existence of a prime $p$ between $r$ and $r+r^{6/11+\varepsilon}$
for arbitrary $\varepsilon>0$ if $r$ is large enough according to
Lou \& Yao \cite{LouYao}.

\begin{corollary}\label{Cor-1}
\begin{compactenum}[\rm (i)]
\item $r\le t(d,r)\le tt(d,r)\le 2r-2$ for all $d\ge1$ and $r\ge2$.
\item $r\le t(d,r)\le tt(d,r)\le (1+o(1))\,r$ for $d\ge1$ and $r\rightarrow\infty$.
\end{compactenum}
\end{corollary}

\begin{proof}
The first, explicit estimate is obtained from Bertrand's postulate: For any given $r$ there is a prime $p$ with $r+1\le p<2r$.
We use $|C_i|\ge 2r-2\ge p-1$ to derive the existence of a topological colored Tverberg $(p-1)$-partition, which in particular yields an $r$-partition since $p-1\ge r$.

The second, asymptotic estimate uses the Lou \& Yao bound instead.
\end{proof}

\begin{remark}
    The colored Tverberg problem as originally posed
    by B\'ar\'any \& Larman \cite{Bar-Lar} in 1992 was
    different from the version we have given above
    (following B\'ar\'any, F\"uredi \& Lov\'asz \cite{BFL} and
    \v{Z}ivaljevi\'c \& Vre\'cica \cite{ZV-1}):
    B\'ar\'any and Larman had asked for an upper bound $N(d,r)$ on
    the cardinality of the union
    $|\CC|$ that together with $|C_i|\ge r$ would force
    the existence of a colored Tverberg $r$-partition.
    This original formulation has two major disadvantages: One is that
    the \v{Z}ivaljevi\'c--Vre\'cica result does not apply to it.
    A second one is that it does not lend itself to estimates
    for the general case in terms of the prime case.

    However, our Corollary \ref{Th-Result3} also solves
    the original version for the case when $r+1$ is a prime.
\end{remark}

The colored Tverberg problem originally arose as a tool to obtain
complexity bounds in computational geometry. As a consequence, our
new bounds can be applied to improve these bounds, as follows. Note
that in some of these results $t(d,d+1)^d$ appears in the exponent,
so even slightly improved estimates on $t(d,d+1)$ have considerable
effect. For surveys see \cite{Bar-1}, \cite[Sect.~9.2]{mat-1}, and
\cite[Sect.~11.4.2]{Ziv:handbook}.

Let $S\subseteq \mathbb{R}^{d}$ be a set in general position of size $n$,
that is, such that no $d+1$ points of $S$ lie on a hyperplane.
Let $h_{d}(n)$
denote the number of hyperplanes that bisect the set $S$ and are spanned by
the elements of the set $S$.
According to B\'{a}r\'{a}ny \cite[p.~239]{Bar-1},
\begin{equation*}
h_{d}(n)=O(n^{d-\varepsilon_{d}})  \qquad\text{with}\qquad
\varepsilon_{d}= t(d,d+1)^{-(d+1)}.
\end{equation*}
Thus we obtain the following bound and equality.

\begin{corollary}
If $d+2$ is a prime then
\begin{equation*}
h_{d}(n)=O(n^{d-\varepsilon_{d}})  \qquad\text{with}\qquad
\varepsilon_{d}= ( d+1) ^{-(d+1)}.
\end{equation*}
For general $d$, we obtain e.g.\ 
$\varepsilon_{d}\geq (d+1) ^{-(d+1)-O(\log d)}$.
\end{corollary}

Let $\CC \subseteq \mathbb{R}^{d}$ be a finite set. A $\CC$\emph{-simplex} is
the convex hull of some collection of $d+1$ points of $\CC $. The second
selection lemma \cite[Thm.~9.2.1]{mat-1}
claims that for an $n$-point set $\CC \subseteq \mathbb{R}^{d}$
and the family $\mathcal{F}$ of $\alpha \binom{n}{d+1}$ $\CC $-simplices
with $\alpha\in(0,1]$ there exists a point contained in at least
$c\cdot\alpha ^{s_{d}}\binom{n}{d+1}$ $\CC $-simplices of $\mathcal{F}$.
Here $c=c(d)>0$ and $s_{d}$ are
constants. For dimensions $d>2$, the presently known proof gives that
$s_{d}\approx t(d,d+1) ^{d+1}$. Again,
Corollary~\ref{Cor-1} yields the following, much better bounds for the
constant~$s_{d}$.

\begin{corollary}
If $d+2>4$ is a prime then the second selection lemma holds for
$s_{d}=(d+1)^{d+1}$, and in general e.g.\ for $s_{d}=(2d+2)^{d+1}$.
\end{corollary}

Let $X\subset \mathbb{R}^{d}$ be an $n$-element set. A
$k$\emph{-facet} of the set $X$ is an oriented $(d-1)$-simplex
$\mathrm{conv}\{x_{1},\dots,x_{d}\}$ spanned by elements of $X$ such
that there are exactly $k$ points of $X$ on its strictly positive
side. When $n-d$ is even, $(\frac{n-d}{2})$-facets of the set $X$ are
called \emph{halving facets}. From \cite[Thm.~11.3.3]{mat-1} we
have a new, better estimate for the number of halving facets.

\begin{corollary}
For $d>2$ and $n-d$ even, the number of halving facets of an $n$-set
$X\subset \mathbb{R}^{d}$ is $O(n^{d-\frac{1}{(2d)^{d}}})$.%
\end{corollary}

\section{The Configuration Space/Test Map scheme}
\label{Sec:CSTM}

According to the ``deleted joins'' version of the general ``Configuration Space/Test Map'' (CS/TM) scheme
for multiple intersection problems,
as pioneered by Sarkaria, Vre\'cica \& \v{Z}ivaljevi\'c, and others,
formalized by \v{Z}ivaljevi\'c, and exposited beautifully by Matou\v{s}ek
\cite[Chap.~6]{MatousekBZ:BU},
we proceed as follows.

Assume that we want to prove the existence of a topological colored Tverberg $r$-partition for an arbitrary colored point set $\CC =C_0\uplus C_1\uplus\dots\uplus C_k$ in $\R^d$ with $|C_i|= t_i$.
So we have to rule out the existence of a (continuous or affine) map
\[
f: C_0 * C_1 * \dots * C_k \ \longrightarrow\ \R^d,
\]
for which every $r$ images of disjoint simplices from the simplicial complex (join of discrete sets) $C_0 * C_1 *\dots * C_k $ have empty intersection in~$\R^d$.
(Compare \v{Z}ivaljevi\'c \cite[Sect.~11.4.2]{Ziv:handbook}.)

The ``deleted joins'' configuration space/test map scheme now suggests taking an $r$-fold deleted join of this map $f$,
where one has to take an $r$-fold $2$-wise deleted join in the domain and an $r$-fold $r$-wise deleted join in the range;
cf.\ \cite[Sect.~6.3]{MatousekBZ:BU}:
\[
 f_{\Delta(2)}^{*r}: \quad (C_0 * C_1 * \dots * C_k)^{*r}_{\Delta(2)}
\ \ \longrightarrow_{\Sym_r}\ \
(\R^d)^{*r}_{\Delta}.
\]
As the join and deleted join operations for simplicial complexes commute \cite[Lemma 6.5.3]{MatousekBZ:BU}, we get the sequence of isomorphisms of simplicial complexes
\begin{align}
( C_{0}\ast C_1\ast \dots\ast C_{k}) _{\Delta (2)}^{\ast r}
&\cong
(C_{0}) _{\Delta (2)}^{\ast r}\ast(C_{1}) _{\Delta (2)}^{\ast r}\ast \dots\ast ( C_{k}) _{\Delta (2)}^{\ast r}\notag\\
&\cong
\Delta _{|C_{0}|,r}\ast \Delta _{|C_{1}|,r}\ast \dots\ast \Delta _{|C_{k}|,r},
\end{align}
where $\Delta_{r,|C_i|}=(C_i)_{\Delta(2)}^{*r}$ is the chessboard complex on $r$ rows and $|C_i|$ columns, on which $\Sym_r$ acts by permuting the $r$ rows.
Thus we arrive at an $\Sym_r$-equivariant map
\begin{equation}
\label{eq:CSTM-scheme-general}
f_{\Delta(2)}^{*r}: \ 
\Delta_{r,|C_0|} * \Delta_{r,|C_1|} * \dots * \Delta_{r,|C_k|}
\ \longrightarrow_{\Sym_r}\ 
(\R^d)^{*r}_{\Delta}\ \subset\ \R^{r\times(d+1)}{\setminus}T
                    \ \simeq \ S(W_r^{\oplus(d+1)}).
\end{equation}
Here
\begin{compactenum}[\rm (i)]
\item\label{b-1}
the simplicial complex $X:= ( C_{0}\ast C_{1}\ast\dots\ast C_{k}) _{\Delta (2)}^{\ast r}\cong\Delta_{r,|C_0|} * \Delta_{r,|C_1|} * \dots * \Delta_{r,|C_k|}$
on the left hand side is an $\Sym_r$-simplicial complex on $r(|C_0| + |C_1| + \dots + |C_k|)$ vertices,
of dimension $|C_0| + |C_1| + \dots + |C_k|-1$ if $|C_i|\le r$ for every $i$,
and of dimension $\min\{|C_0|,r\}+\min\{|C_1|,r\}+\dots+\min\{|C_k|,r\}-1$ in general.
 
Points in $X$ can be represented as convex combinations $\lambda_1 x_1+ \dots + \lambda_r x_r$, where $x_i$ is a point in
(a simplex of) the $i$-th ``join component'' of the iterated deleted join $(C_0 * C_1 * \dots * C_k)^{*r}_{\Delta(2)}$, with
$\lambda_i\ge0$ for all $i$ and $\sum_i\lambda_i=1$.
 
\item\label{b-2}
$(\R^d)^{*r}_{\Delta}:=\{\alpha_1y_1+\cdots+\alpha_ry_r\in(\R^d)^{*r}:\alpha_i\geq 0,\sum_i\alpha_i=1\}{\setminus}\{\tfrac{1}{r}y+\cdots+\tfrac{1}{r}y:y\in\R^d\}$
is a deleted join, which $\Sym_r$-equivariantly embeds into the space of all real $r\times(d+1)$-matrices for which not
all rows are equal, and where $\Sym_r$ acts by permuting the rows.
The diagonal $T$ is the $(d+1)$-dimensional subspace of all matrices for which all rows are equal.
To project on the orthogonal complement of the diagonal $T$ we subtract from each row the average of all the rows.
This operation yields an $\Sym_r$-equivariant orthogonal projection to $W_r^{\oplus(d+1)}{\setminus}\{0\}$, the space of all real $r\times(d+1)$-matrices with column sums equal to zero
but for which not all rows are zero, and where $\Sym_r$ still acts by permuting the rows.
This in turn is homotopy equivalent to the sphere $S(W_r^{\oplus(d+1)})=(S^{r-2})^{*(d+1)}=S^{(r-1)(d+1)-1}=S^{N-1}$, where $\pi\in\Sym_r$ reverses the orientation exactly
if $(\mathrm{sgn\,}\pi)^{d+1}$ is negative.
 
\item\label{b-3}
The action of $\Sym_r$ is non-free exactly on the subcomplex $A\subset X =( C_{0}\ast C_1\ast \dots\ast C_{k}) _{\Delta (2)}^{\ast r}$
given by all the points $\lambda_1 x_1+ \dots + \lambda_r x_r \in ( C_{0}\ast C_1\ast\dots\ast C_{k}) _{\Delta (2)}^{\ast r}$ such that $\lambda_i=\lambda_j=0$ for two distinct indices $i<j$.
These lie in simplices that have no vertices in the $i$-th and $j$-th ``join component'' of the iterated deleted join $(C_0 * C_1\ast\dots * C_k)^{*r}_{\Delta(2)}$,
so the transposition $\pi_{ij}:=(ij)=\binom{\ldots i\ldots j\ldots}{\ldots j\ldots i\ldots}$ fixes these simplices pointwise.
 
\item\label{b-4}
The map $f_{\Delta(2)}^{*r}:X\rightarrow\R^{r\times(d+1)}$ suggested by the ``deleted joins'' scheme takes the point
$\lambda_1 x_1+ \dots + \lambda_r x_r$ and maps it to the $r\times(d+1)$-matrix in $\R^{r\times(d+1)}$ whose $\ell$-th row is 
\[
\big(\lambda_{\ell},\lambda_{\ell} f(x_{\ell})\big).
\]
For an arbitrary map $f$, the image of $A$ under $f^{*r}_{\Delta(2)}$ does not intersect the diagonal $T$:
If $\lambda_i=\lambda_j=0$, then not all rows $(\lambda_{\ell},\lambda_{\ell}f(x_{\ell}))$ can be equal, since $\sum_{\ell}\lambda_{\ell}=1$.
 
However, for the following we replace $f_{\Delta(2)}^{*r}$ by the map $F_0:X\rightarrow\R^{r\times(d+1)}$
that maps $\lambda_1 x_1+ \dots + \lambda_r x_r$ to the $r\times(d+1)$-matrix whose $\ell$-th row is 
\[
\big(\lambda_{\ell},(\prod_{h=1}^r\lambda_h)f(x_{\ell})\big).
\]
The two maps $ f_{\Delta(2)}^{*r}$ and $F_0$ are homotopic as maps
 $A\rightarrow\R^{r\times(d+1)}\setminus T$ by a linear homotopy, so the resulting extension problems are equivalent by \cite[Prop.~3.15(ii)]{Dieck87}.
The advantage of the map $F_0$ is that its restriction to $A$ is independent of $f$.
Indeed, for $\lambda_1 x_1+ \dots + \lambda_r x_r\in A$ and any map $f$ the corresponding $F_0$-image is the $r\times(d+1)$-matrix whose $\ell$-th row is $(\lambda_{\ell},0)$.
\end{compactenum}

Thus we have established the following.

\begin{proposition}
    [CS/TM scheme for the generalized topological colored Tverberg problem]
    \label{prop:CSTM-scheme-general}%
If for some parameters $(d,r,k;t_0,\dots,t_k)$
an $\Sym_r$-equivariant extension \mbox{\rm(\ref{eq:CSTM-scheme-general})}
of the map $F_0|_A:A\rightarrow\ \R^{r\times(d+1)}{\setminus}T$
does not exist, then a topological colored Tverberg
\mbox{$r$-partition} exists for all
continuous $f: C_0 * C_1 * \dots * C_k \rightarrow\R^d$ with $|C_i|\ge t_i$ for all~$i$.
\end{proposition}

\v{Z}ivaljevi\'c \& Vre\'cica in \cite{ZV-1} achieved this for $(d,r,d;2r-1,\dots,2r-1)$
and prime $r$ by applying a Borsuk--Ulam type theorem
to the action of the cyclic subgroup $\mathbb{Z}_r\subset\Sym_r$,
which acts freely on the join of chessboard complexes if $r$ is a prime.
However, they lose a factor of $2$ from the fact that the chessboard complexes $\Delta_{r,t}$, for $r\leq t$, of dimension $r-1$
are homologically $(r-2)$-connected only if $t\ge 2r-1$;
compare \cite{BLZV}, \cite{Zie:chess}, and \cite{Sa-Wa}.

Our Theorem~\ref{thm:main} claims this for $(d,r,d+1;r-1,\ldots,r-1,1)$.
To prove it, we will use relative equivariant obstruction theory,
as presented
by tom Dieck in \cite[Sect.~II.3]{Dieck87}.

\section{Proof of Theorem \protect\ref{thm:main}}

First we establish that the scheme of Proposition~\ref{prop:CSTM-scheme-general} fails if applied to the colored Tverberg problem $(d,r,d;r,\ldots,r)$
associated with the B\'ar\'any--Larman conjecture directly.

\begin{proposition}\label{prop:fails}
For all $r\ge2$ and $d\ge1$, with $N=(r-1)(d+1)$, an $\Sym_r$-equivariant map
    \[
    F: (\Delta_{r,r})^{*(d+1)}
    \ \ \longrightarrow_{\Sym_r}\ \ W_r^{\oplus(d+1)}\setminus\{0\}
    \ \simeq\ S^{N-1}
    \]
exists.
\end{proposition}

\begin{proof}
Let $M:=r(d+1)-1$, and let $\Delta_M$ be an $M$-dimensional simplex whose vertex set $\CC =C_0\uplus C_1\uplus\dots\uplus C_d$ is colored by $d+1$ colors such that $|C_i|=r$ for every $i$.
For an arbitrary continuous map $f:\Delta_M\to\R^d$ the ``deleted join'' configuration space/test maps scheme, as in \eqref{eq:CSTM-scheme-general},
induces an $\Sym_r$-equivariant map $F$:
\[
(C_0 * C_1 * \dots * C_d)^{*r}_{\Delta(2)}\cong (\Delta_{r,r})^{*(d+1)}  \ \ \longrightarrow_{\Sym_r}\R^{r\times (d+1)} \ \ \longrightarrow_{\Sym_r}\ \ W_r^{\oplus(d+1)},
\]
where the second map is the projection on the orthogonal complement of the diagonal.

The $\Sym_r$-action on the configuration space $(\Delta_{r,r})^{*(d+1)}$ is not free; let $A$ denote the subcomplex of $(\Delta_{r,r})^{*(d+1)}$ on which $\Sym_r$ does not act freely.
As we have seen in Section~\ref{Sec:CSTM}, item \eqref{b-4}, the $F$-image of $A$ avoids the origin in  $W_r^{\oplus(d+1)}$.

For any facet of the $(r-1)$-dimensional chessboard complex $\Delta_{r,r}$ there is an elementary collapse which removes the facet together with its subfacet (ridge of the chessboard complex)
obtained by deleting the vertex in the $r$-th column.
Performing these collapses simultaneously, we see that $\Delta_{r,r}$ collapses $\Sym_r$-equivariantly to an $(r-2)$-dimensional subcomplex of $\Delta_{r,r}$, and thus
$(\Delta_{r,r})^{*(d+1)}$ equivariantly retracts to a subcomplex $X\subset (\Delta_{r,r})^{*(d+1)}$ whose dimension is only $(d+1)(r-1)-1=N-1$.
Now it is enough to construct an $\Sym_r$-equivariant map 
\[
   X \longrightarrow_{\Sym_r} S(W_r^{\oplus (d+1)})=S^{N-1}.
\]

Note that the $\Sym_r$-action on $X$ is not free:
The subcomplex of $X$ on which $\Sym_r$ does not act freely is $X\cap A$.
Since $\dim X=\dim S^{N-1}$, $S^{N-1}$ is $(N-1)$-simple and $S^{N-1}$ is $(N-2)$-connected, by relative equivariant obstruction theory,
there is no obstruction for the existence of an $\Sym_r$-equivariant map $X \longrightarrow_{\Sym_r} S^{N-1}$ provided that an $\Sym_r$-equivariant map
$X\cap A\longrightarrow_{\Sym_r} S^{N-1}$ on the non-free part of the domain can be exhibited.

Since the $F$-image of $A$ avoids the origin in  $W_r^{\oplus(d+1)}$ the restriction $F|_{X\cap A}$ composed with the $\Sym_r$-equivariant radial projection to the sphere induces the required
\mbox{$\Sym_r$-equivariant} map $X\cap A\longrightarrow_{\Sym_r} W_r^{\oplus(d+1)}{\setminus}\{0\}\longrightarrow_{\Sym_r} S^{N-1}$.
\end{proof}

We now specialize the general scheme of Proposition~\ref{prop:CSTM-scheme-general} to the situation of Theorem~\ref{thm:main}.
Let $[n]:=\{1,\ldots,n\}$ denote the $0$-dimensional simplicial complex on $n$ vertices.
Then we have to show the following.

\begin{proposition}\label{prop:main}
 Let $r\geq 2$ and $d\geq 1$ be integers, and $N=(r-1)(d+1)$. 

An $\mathfrak{S}_{r}$-equivariant map
\begin{equation*}
F:(\Delta _{r,r-1})^{\ast d}\ast \Delta _{r,r-1}\ast \lbrack r]\ \
\longrightarrow _{\mathfrak{S}_{r}}\ \ W_{r}^{\oplus (d+1)}\setminus
\{0\}
\end{equation*}
that extends the equivariant map $F_0|_A$
from {\rm Section~\ref{Sec:CSTM}, item~(\ref{b-4})}, exists if and only if
\begin{equation*}
r\mid(r-1)!^{d}.
\end{equation*}
\end{proposition}

The vertex set of the join $(\Delta_{r,r-1})^{*d} * \Delta_{r,r-1} * [r]$ may be represented by a rectangular array of size $r\times((r-1)(d+1)+1)$,
which carries the $d+1$ chessboard complexes $\Delta_{r,r-1}$ lined up from left to right, and in the last column has the chessboard complex $\Delta_{r,1}=[r]$,
which is just a discrete set.
(See Figure~\ref{fig:array}.)

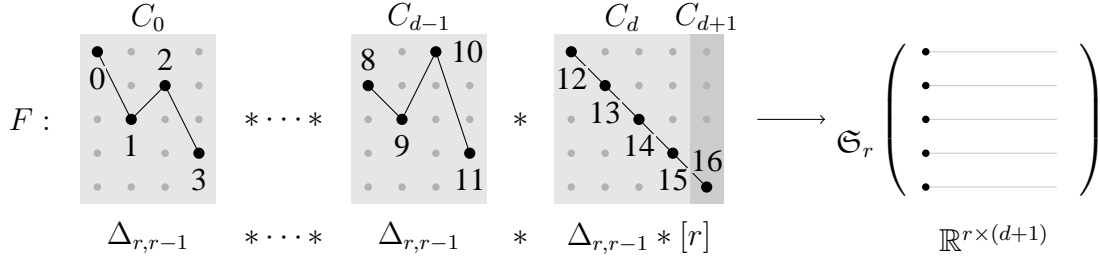
\begin{figure}[ht!]
\begin{center}
\medskip

 \begin{tikzpicture}
\def\a{0.45} 			

\node at (-2*\a,2*\a){$F:$};

\node at (1.5*\a, 5*\a) {$C_0$};
\fill[black!10] (-\a/2,-\a/2) rectangle (3.5*\a,4.5*\a);
\foreach \x in {0,...,3}
  \foreach \y in {0,...,4}
    \node[circle,inner sep= 1 pt,fill, black!30] at (\a*\x,\a*\y) {};		
\draw (0,4*\a) node[circle, fill, inner sep=1.5pt,label={[fill=black!10, inner sep=0.3pt, label distance=3 pt]below:0}] {} 		
      -- (\a, 2*\a) node[circle, fill, inner sep=1.5pt, label={[fill=black!10, inner sep=0.3pt, label distance=3 pt]below:1}]{}
      -- (2*\a, 3*\a) node[circle, fill, inner sep=1.5pt, label={[fill=black!10, inner sep=0.3pt, label distance=3 pt]above:2}]{}
      -- (3*\a, \a) node[circle, fill, inner sep=1.5pt,label={[fill=black!10, inner sep=0.3pt, label distance=3 pt]below:3}]{};
\node at (1.5*\a, -1.5*\a) {$\Delta_{r,r-1}$};

\node at (5.5*\a,2*\a){$* \cdots *$};
\node at (5.5*\a,-1.5*\a){$* \cdots *$};

\node at (9.5*\a, 5*\a) {$C_{d-1}$};
\fill[black!10] (7.5*\a,-\a/2) rectangle (11.5*\a,4.5*\a);
\foreach \x in {0,...,3}
  \foreach \y in {0,...,4}
    \node[circle,inner sep= 1 pt,fill, black!30] at (8*\a+\x*\a,\a*\y) {};	
\draw (8*\a,3*\a) node[circle, fill, inner sep=1.5pt, label={[fill=black!10, inner sep=0.3pt, label distance=3 pt]above:8}]{} 		
      -- (9*\a, 2*\a) node[circle, fill, inner sep=1.5pt,label={[fill=black!10, inner sep=0.3pt, label distance=3 pt]below:9}]{}
      -- (10*\a, 4*\a) node[circle, fill, inner sep=1.5pt,label={[fill=black!10, inner sep=0.3pt, label distance=3 pt]right:10}]{}
      -- (11*\a, \a) node[circle, fill, inner sep=1.5pt,label={[fill=black!10, inner sep=0.3pt, label distance=3 pt]below:11}]{};
\node at (9.5*\a, -1.5*\a) {$\Delta_{r,r-1}$};

\node at (12.5*\a,2*\a){$*$};
\node at (12.5*\a,-1.5*\a){$*$};

\node at (15.5*\a, 5*\a) {$C_d$};
\node at (18*\a, 5*\a) {$C_{d+1}$};
\fill[black!10] (13.5*\a,-\a/2) rectangle (17.5*\a,4.5*\a);
\fill[black!20] (17.5*\a,-\a/2) rectangle (18.5*\a,4.5*\a);
\foreach \x in {0,...,4}
  \foreach \y in {0,...,4}
    \node[circle,inner sep= 1 pt,fill, black!30] at (14*\a+\x*\a,\a*\y) {};	
\draw (14*\a,4*\a) node[circle, fill, inner sep=1.5pt, label={[fill=black!10, inner sep=0.3pt, label distance=3 pt]below:12}]{} 	
      -- (15*\a, 3*\a) node[circle, fill, inner sep=1.5pt, label={[fill=black!10, inner sep=0.3pt, label distance=3 pt]below:13}]{}
      -- (16*\a, 2*\a) node[circle, fill, inner sep=1.5pt, label={[fill=black!10, inner sep=0.3pt, label distance=3 pt]below:14}]{}
      -- (17*\a, \a) node[circle, fill, inner sep=1.5pt, label={[fill=black!10, inner sep=0.3pt, label distance=3 pt]below:15}]{}
      -- (18*\a, 0) node[circle, fill, inner sep=1.5pt, label={[fill=black!20, inner sep=0.3pt, label distance=3 pt]above:16}]{};
\node at (16*\a, -1.5*\a) {$\Delta_{r,r-1}*[r]$};

\draw[->] (19.5*\a, 2*\a)--(21.5*\a,2*\a) node[below right] {$\mathfrak{S}_r$};

\matrix [right,left delimiter=(,right delimiter=)] at (24*\a,2*\a) {
  \foreach \y in {0,...,4}
    \draw[black!20] (0,\a*\y) node[circle, fill, inner sep=1pt, black] {} -- (4*\a,\a*\y) node[circle, fill, inner sep=1pt, opacity=0] {};\\ } ;
\node at (26.5*\a, -1.5*\a) {$\mathbb{R}^{r \times (d+1)}$};

\end{tikzpicture}
\end{center}
\caption{The vertex set, and one facet in $\Phi$ of the combinatorial configuration space for $r=5$.}
\label{fig:array}
\end{figure}

The join of chessboard complexes
$(\Delta_{r,r-1})^{*d} * \Delta_{r,r-1} * [r]$
has the dimension\break
$(r-1)(d+1)=N$, while the target sphere has dimension $N-1$.
On both of them, $\Sym_r$ acts by
permuting the rows.

While the chessboard complexes $\Delta_{r,r}$ collapse equivariantly to
lower-dimensional complexes, the chessboard complexes $\Delta_{r,r-1}$
are closed oriented pseudomanifolds of dimension $r-2$ and thus
don't collapse;
for example, $\Delta_{3,2}$ is a circle and $\Delta_{4,3}$ is a torus.
We will read the maximal simplices of such a complex from left to right,
which yields the orientation cycle in a special form
with few signs that will be very convenient.

\begin{lemma} \emph{(cf.\ \cite{BLZV}, \cite{Sa-Wa}, \cite[p.~145]{Jonsson})}
\label{Lemma:Chess-Manifold}%
For $r>2$, the chessboard complex $\Delta_{r,r-1}$
is a connected, orientable pseudomanifold of dimension $r-2$.
Therefore
\begin{equation*} 
{H}_{r-2}(\Delta_{r,r-1};\mathbb{Z)=Z}
\end{equation*}
and an orientation cycle is
\begin{equation}
z_{r,r-1}\ =\ \sum_{\pi\in\Sym_{r}}
(\mathrm{sgn\,}\pi)
\langle(\pi(1),1),\dots,(\pi(r-1),r-1)\rangle.
  \label{eq:generating_cocycle}
\end{equation}
The group $\Sym_r$ acts on $\Delta_{r,r-1}$ by permuting the rows;
this affects the orientation according to
$\pi\cdot z_{r,r-1}=(\mathrm{sgn\,}\pi) z_{r,r-1}$.
\end{lemma}

\begin{proof}[Proof of Proposition~\ref{prop:main}]
    For $r=2$, since $2\nmid 1$, this says that there is no equivariant map
    $S^N\rightarrow S^{N-1}$,
    where both spheres are equipped with the antipodal action:
    This is the Borsuk--Ulam theorem (and the Lov\'asz proof).
    Thus we may now assume that $r\ge3$.

Let $X:=(\Delta_{r,r-1})^{*(d+1)}*[r]$
be our combinatorial configuration space,
$A\subset X$ the non-free subcomplex, and
$F_0:A\rightarrow_{\Sym_r} S(W_{r}^{\oplus (d+1)})$
the prescribed map that we are to extend $\Sym_r$-equivariantly
to~$X$.

Since
\begin{compactitem}
\item $\dim X=N$ and $\dim S(W_{r}^{\oplus (d+1)})=N-1$, with
\item $\mathrm{conn\,}S(W_{r}^{\oplus (r+1)})=N-2$, and
\item $S(W_{r}^{\oplus (r+1)})$ being $(N-2)$-simple,
\end{compactitem}
by \cite[Sect.~II.3]{Dieck87} the existence of an $\Sym_r$-equivariant extension of the map 
\[F_0:A\rightarrow_{\Sym_r} S(W_{r}^{\oplus (d+1)})
\]
to an $\Sym_r$-equivariant map 
\[
X \rightarrow_{\Sym_r} S(W_{r}^{\oplus (d+1)})
\] 
is equivalent to the vanishing of the primary obstruction
\[
\mathfrak{o}\ \in \
H_{\Sym_r}^{N}\big(X,A ;\pi_{N-1}(S(W_{r}^{\oplus (d+1)}))\big).
\]
The Hurewicz isomorphism gives an isomorphism of the
coefficient $\Sym_r$-module with a homology group,
\begin{equation*}
\pi_{N-1}(S(W_{r}^{\oplus (r+1)}))\ \cong\ H_{N-1}(S(W_{r}^{\oplus
(r+1)}); \mathbb{Z})\ \ =:\ \ \mathcal{Z}.
\end{equation*}
As an abelian group this module $\mathcal{Z}=\langle \zeta \rangle $
is isomorphic to $\mathbb{Z}$.
The action of the permutation $\pi \in \Sym_r$
on the module $\mathcal{Z}$ is given by
\begin{equation*}
\pi \cdot\zeta\ =\ (\mathrm{sgn\,}\pi ) ^{d+1}\zeta.
\end{equation*} 
 
\noindent
\textbf{Computing the obstruction cocycle.}
We will now compute an obstruction cocycle $\mathfrak{c}_f$ in the cochain group $C_{\Sym_r}^{N}\big( X,A ;\mathcal{Z}\big)$.
Then we show that the cocycle $\mathfrak{c}_f$ is not a coboundary (that is, it does not vanish when passing to $\mathfrak{o}=[\mathfrak{c}_f]$
in the cohomology group $H_{\Sym_r}^{N}( X,A ;\mathcal{Z})$) if and only if $r\nmid (r-1)!^d$.

For this, we use a specific general position map $f:\Delta_{N}\rightarrow\R^d$, which induces a map $F:X\rightarrow\R^{r\times(d+1)}$;
the value of the obstruction cocycle $\mathfrak{c}_f$ on an oriented maximal simplex $\sigma$ of $X$ is then given by the signed intersection number of $F(\sigma)$ with the
test space, the diagonal $T$, or in other words by the mapping degree 
\[
 \deg (F|_{\partial\,\sigma}:\partial\,\sigma\to S(W_{r}^{\oplus (d+1)})).
\]
(Compare \cite{Dieck87} and \cite{Sale}.)

Let $e_{1},\dots,e_{d}$ be the standard basis vectors
of $\R^d$, set $e_{0}:=0\in\R^d$, and denote
by $v_{0},\dots,v_{N}$ the set of vertices of the $N$-simplex $\Delta_N$
in the given order, that is, such that
$C_i=\{v_{i(r-1)},\dots,v_{(i+1)(r-1)-1}\}$ for $i\le d$
and $C_{d+1}=\{v_{(d+1)(r-1)}\}$.
Let $f:\Delta_{N} \rightarrow\R^d$ be the affine map
defined on the vertices by
\begin{equation*}
\left\{
\begin{array}{llll}
v_{i} & \overset{f}{\longmapsto } & e_{\lfloor {i}/({r-1})\rfloor } , &
\text{for }0\leq i\le N-1, \\
v_{N} & \overset{f}{\longmapsto } & \tfrac{1}{d+1}\sum_{i=0}^{d}e_{i},
\end{array}
\right.
\end{equation*}
that is, such that the vertices in $C_i$ are mapped to
the vertex $e_i$ of the standard $d$-simplex for $i\le d$,
while $v_N\in C_{d+1}$ is mapped to the center of this simplex.

\begin{figure}[ht!]
\begin{center}
\medskip
\begin{tikzpicture}

\def\b{1.6};			
\draw    (0, 2.51*\b) node[circle, fill, inner sep=1.5pt, label=left:\small $v_0$] {}
      -- (0.1*\b, 2.1*\b) node[circle, fill, inner sep=1.5pt, label=left:\small $v_1$] {}
      -- (0.3*\b, 1.8*\b)node[circle, fill, inner sep=1.5pt, label=left:\small $v_2$] {}
      -- (0.4*\b, 1.4*\b)node[circle, fill, inner sep=1.5pt, label=left:\small $v_3$] {};
\draw[opacity=0] (0.4*\b, 1.4*\b)
      -- (1.4*\b, 1*\b) node[sloped, midway,opacity=1] {$\hdots$};
\draw    (1.4*\b, 1*\b) node[circle, fill, inner sep=1.5pt, label=right:\small $v_{14}$] {}
      -- (1.5*\b, 0.7*\b)node[circle, fill, inner sep=1.5pt, label=left:\small $v_{15}$] {}
      -- (2.1*\b, 0.35*\b)node[circle, fill, inner sep=1.5pt, label=right:\small $v_{16}$] {};
\draw[->] (3*\b,1.75*\b) -- (4*\b,1.75*\b);
\node[shift={(5*\b,\b)}, outer sep=0, inner sep=0] at (0, 0) (a) {};
\node[shift={(5*\b,\b)}, outer sep=0, inner sep=0] at (1.25*\b, 1.51*\b) (b) {};
\node[shift={(5*\b,\b)}, outer sep=0, inner sep=0] at (2.78*\b, 0) (c) {};
\node[shift={(5*\b,\b)}, outer sep=0, inner sep=0] at (1.80*\b, -0.69) (d) {};
\node[shift={(5*\b,\b)},circle, fill, gray, inner sep=1.5pt,draw, label={[darkgray, label distance=-2pt]left:\small $f(v_{16})$}] at (1.30*\b, 0.85) (e) {};
\draw[dotted] (a) -- (b) -- (c) -- (a) -- (d) -- (c);
\draw[dotted]  (d) -- (b);
\node[circle, fill, inner sep= 1.5, below of=a, node distance=5pt, label= {below,label distance=-2pt}:\small $f(v_0)$] {};
\node[circle, fill, inner sep= 1.5, left of=a, node distance=5pt, label= {left,label distance=-2pt}:\small $f(v_1)$] {};
\node[circle, fill, inner sep= 1.5, above of=a, node distance=5pt, label= {above,label distance=-2pt}:\small $f(v_2)$] {};
\node[circle, fill, inner sep= 1.5, right of=a, node distance=5pt, label= {right,label distance=-2pt}:\small $f(v_3)$] {};
\node[circle, fill, inner sep= 1.5, below of=b, node distance=5pt, label= {below,label distance=-2pt}:\small $f(v_4)$] {};
\node[circle, fill, inner sep= 1.5, left of=b, node distance=5pt, label= {left,label distance=-2pt}:\small $f(v_5)$] {};
\node[circle, fill, inner sep= 1.5, above of=b, node distance=5pt, label= {above,label distance=-2pt}:\small $f(v_6)$] {};
\node[circle, fill, inner sep= 1.5, right of=b, node distance=5pt, label= {right,label distance=-2pt}:\small $f(v_7)$] {};
\node[circle, fill, inner sep= 1.5, below of=c, node distance=5pt, label= {below,label distance=-2pt}:\small $f(v_8)$] {};
\node[circle, fill, inner sep= 1.5, left of=c, node distance=5pt, label= {left,label distance=-2pt}:\small $f(v_{9})$] {};
\node[circle, fill, inner sep= 1.5, above of=c, node distance=5pt, label= {above,label distance=-2pt}:\small $f(v_{10})$] {};
\node[circle, fill, inner sep= 1.5, right of=c, node distance=5pt, label= {right,label distance=-2pt}:\small $f(v_{11})$] {};
\node[circle, fill, inner sep= 1.5, below of=d, node distance=5pt, label= {below,label distance=-2pt}:\small $f(v_{12})$] {};
\node[circle, fill, inner sep= 1.5, left of=d, node distance=5pt, label= {left,label distance=-2pt}:\small $f(v_{13})$] {};
\node[circle, fill, inner sep= 1.5, above of=d, node distance=5pt, label= {above,label distance=-2pt}:\small $f(v_{14})$] {};
\node[circle, fill, inner sep= 1.5, right of=d, node distance=5pt, label= {right,label distance=-2pt}:\small $f(v_{15})$] {};

\end{tikzpicture}

\end{center}
\caption{The map $f:\Delta ^{16} \rightarrow \mathbb{R}^{3}$ in the case $d=3$ and $r=5$.}
\end{figure}
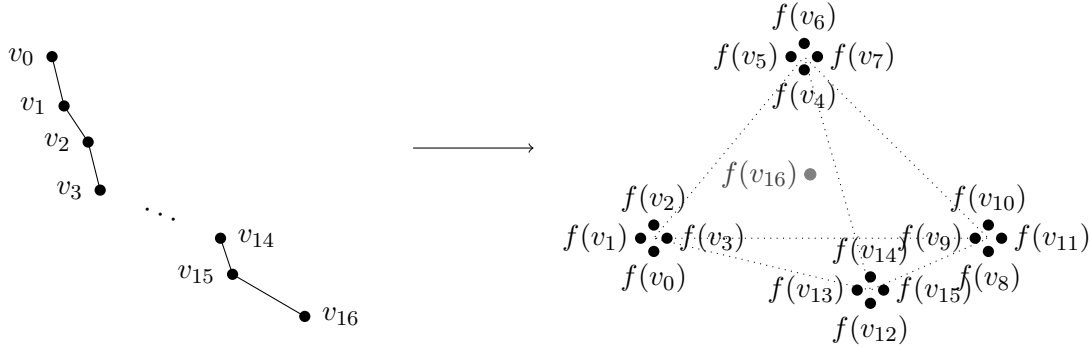

This induces an affine map $f: C_0 * C_1 *\dots * C_{d+1}\rightarrow \R^d$ and thus an equivariant map $F:X\rightarrow\R^{r\times(d+1)}$,
taking $\lambda_1x_1+\dots+\lambda_rx_r$ to the $r\times(d+1)$-matrix whose $\ell$-th row is $(\lambda_{\ell},(\prod_{h=1}^r\lambda_h)x_\ell)$,
which extends the prescribed map $F_0:A\rightarrow\R^{r\times(d+1)}{\setminus}T$.
The intersection points of the image of $F$ with the diagonal $T$ correspond to the topological colored Tverberg $r$-partitions of the
configuration $\CC=C_0\uplus\dots\uplus C_{d+1}$ in~$\R^d$.
Since $\lambda_1=\dots=\lambda_r=\tfrac1r$ at all these intersection points, we find that $F$ is in general position with respect to $T$.

The only Tverberg $r$-partitions of the point configuration $\CC$
(even ignoring colors) are given by $r-1$ $d$-simplices with its
vertices at $e_0,e_1,\dots,e_d$, together with one singleton point
($0$-simplex) at the center. Clearly there are $(r-1)!^d$ such
partitions.

We take representatives for the $\Sym_r$-orbits of maximal simplices of~$X$ 
such that from the last $\Delta_{r,r-1}$  factor,
the vertices $(1,1),\dots,(r-1,r-1)$ are taken.

On the simplices of $X$ we use the orientation that is induced by ordering all vertices left-to-right on the array of Figure~\ref{fig:array}.
This orientation is $\Sym_r$-invariant, as permutation of the rows does not affect the left-to-right ordering.

\smallskip

\noindent
\textbf{The obstruction cocycle evaluated on subcomplexes of $X$.}
Let us consider the following chains of dimensions $N$ resp.~$N-1$
(illustrated in Figure~\ref{fig:chains}),
where $z_{r,r-1}$ denotes the orientation cycle for the
chessboard complex $\Delta_{r,r-1}$,
as given by Lemma~\ref{Lemma:Chess-Manifold}:
\begin{eqnarray*}
\Phi     &=& (z_{r,r-1})^{*d}*
\langle (1,1),\dots,         \ \dots\ ,\dots,(r-1,r-1),(r,r)\rangle
,\\
\Omega_j &=& (z_{r,r-1})^{*d}*
\langle (1,1),\dots,         \ \ldots\ ,\dots,(r-1,r-1),(j,r)\rangle,
\qquad (1\le j < r)
,\\
\Theta_i &=& (z_{r,r-1})^{*d}*
\langle (1,1),\dots,\widehat{(i,i)},\dots,(r-1,r-1),(r,r)\rangle,
\qquad (1\le i\le r),\\
\Theta_{i,j} &=& (z_{r,r-1})^{*d}*
\langle (1,1),\dots,\widehat{(i,i)},\dots,(r-1,r-1),(j,r)\rangle,
\qquad (1\le i,j< r). 
\end{eqnarray*}
Here we use the usual notation
$\langle w_0,\dots,\widehat{w_i},\dots,w_k\rangle$
for an oriented simplex with
ordered vertex set $(w_0,\dots,\widehat{w_i},\dots,w_k)$ from which the vertex $ w_i$ is
omitted.
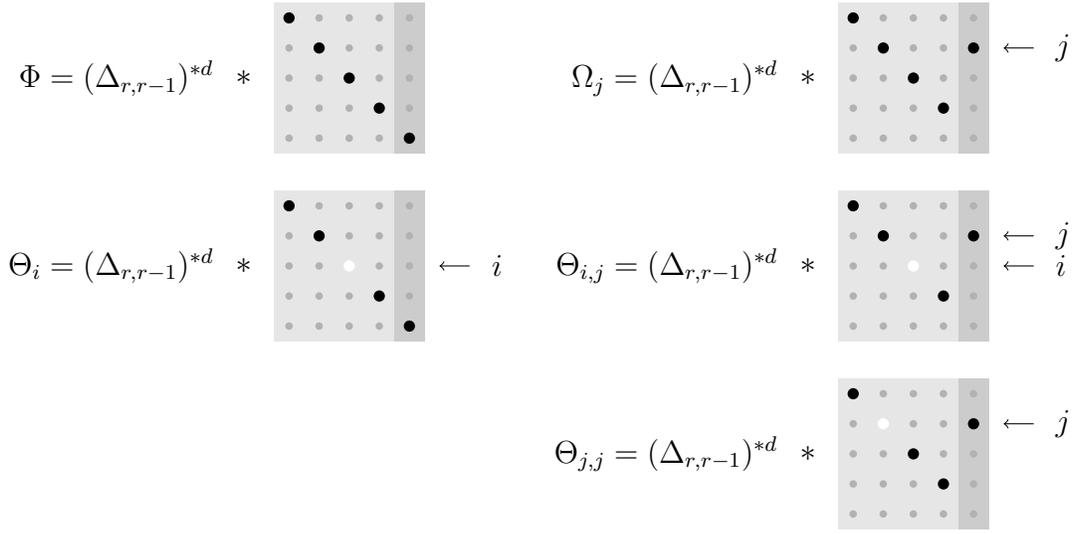
\begin{figure}[ht!]
\begin{center}
\medskip

\begin{tikzpicture}
\def\a{0.4}			
\def\X{7.5}			
\def\Y{-2.5}			

\begin{scope}[shift={(0,0)}] 
\node[left, outer sep=0, inner sep=0] at (-2.5*\a,2*\a){$\Phi=(\Delta_{r,r-1})^{*d}$};
\node at (-1.5*\a,2*\a){$*$};
\fill[black!10] (-0.5*\a,-\a/2) rectangle (3.5*\a,4.5*\a);
\fill[black!20] (3.5*\a,-\a/2) rectangle (4.5*\a,4.5*\a);
\foreach \x in {0,...,4}
  \foreach \y in {0,...,4}
    \node[circle,inner sep= 1 pt,fill, black!30] at (\x*\a,\a*\y) {};
\node[circle, fill, inner sep=1.5pt] at (0,4*\a) {};
\node[circle, fill, inner sep=1.5pt] at (1*\a, 3*\a) {};
\node[circle, fill, inner sep=1.5pt] at (2*\a, 2*\a) {};
\node[circle, fill, inner sep=1.5pt] at (3*\a, \a) {};
\node[circle, fill, inner sep=1.5pt] at (4*\a, 0) {};
\end{scope}

\begin{scope}[shift={(0,\Y)}] 
\node[left, outer sep=0, inner sep=0] at (-2.5*\a,2*\a){$\Theta_i=(\Delta_{r,r-1})^{*d}$};
\node at (-1.5*\a,2*\a){$*$};
\fill[black!10] (-0.5*\a,-\a/2) rectangle (3.5*\a,4.5*\a);
\fill[black!20] (3.5*\a,-\a/2) rectangle (4.5*\a,4.5*\a);
\foreach \x in {0,...,4}
  \foreach \y in {0,...,4}
    \node[circle,inner sep= 1 pt,fill, black!30] at (\x*\a,\a*\y) {};
\node[circle, fill, inner sep=1.5pt] at (0,4*\a) {};
\node[circle, fill, inner sep=1.5pt] at (1*\a, 3*\a) {};
\node[circle, fill, white, inner sep=1.5pt] at (2*\a, 2*\a) {};
\node[circle, fill, inner sep=1.5pt] at (3*\a, \a) {};
\node[circle, fill, inner sep=1.5pt] at (4*\a, 0) {};
\draw[<-] (5*\a, 2*\a) -- (6*\a, 2*\a) node[label=right:$i$] {};
\end{scope}

\begin{scope}[shift={(\X,0)}] 
\node[left, outer sep=0, inner sep=0] at (-2.5*\a,2*\a){$\Omega_j=(\Delta_{r,r-1})^{*d}$};
\node at (-1.5*\a,2*\a){$*$};
\fill[black!10] (-0.5*\a,-\a/2) rectangle (3.5*\a,4.5*\a);
\fill[black!20] (3.5*\a,-\a/2) rectangle (4.5*\a,4.5*\a);
\foreach \x in {0,...,4}
  \foreach \y in {0,...,4}
    \node[circle,inner sep= 1 pt,fill, black!30] at (\x*\a,\a*\y) {};
\node[circle, fill, inner sep=1.5pt] at (0,4*\a) {};
\node[circle, fill, inner sep=1.5pt] at (1*\a, 3*\a) {};
\node[circle, fill, inner sep=1.5pt] at (2*\a, 2*\a) {};
\node[circle, fill, inner sep=1.5pt] at (3*\a, \a) {};
\node[circle, fill, inner sep=1.5pt] at (4*\a, 3*\a) {};
\draw[<-] (5*\a, 3*\a) -- (6*\a, 3*\a) node[label=right:$j$] {};
\end{scope}

\begin{scope}[shift={(\X,\Y)}] 
\node[left, outer sep=0, inner sep=0] at (-2.5*\a,2*\a){$\Theta_{i,j}=(\Delta_{r,r-1})^{*d}$};
\node at (-1.5*\a,2*\a){$*$};
\fill[black!10] (-0.5*\a,-\a/2) rectangle (3.5*\a,4.5*\a);
\fill[black!20] (3.5*\a,-\a/2) rectangle (4.5*\a,4.5*\a);
\foreach \x in {0,...,4}
  \foreach \y in {0,...,4}
    \node[circle,inner sep= 1 pt,fill, black!30] at (\x*\a,\a*\y) {};
\node[circle, fill, inner sep=1.5pt] at (0,4*\a) {};
\node[circle, fill, inner sep=1.5pt] at (1*\a, 3*\a) {};
\node[circle, fill, white, inner sep=1.5pt] at (2*\a, 2*\a) {};
\node[circle, fill, inner sep=1.5pt] at (3*\a, \a) {};
\node[circle, fill, inner sep=1.5pt] at (4*\a, 3*\a) {};
\draw[<-] (5*\a, 3*\a) -- (6*\a, 3*\a) node[label=right:$j$] {};
\draw[<-] (5*\a, 2*\a) -- (6*\a, 2*\a) node[label=right:$i$] {};
\end{scope}

\begin{scope}[shift={(\X,2*\Y)}] 
\node[left, outer sep=0, inner sep=0] at (-2.5*\a,2*\a){$\Theta_{j,j}=(\Delta_{r,r-1})^{*d}$};
\node at (-1.5*\a,2*\a){$*$};
\fill[black!10] (-0.5*\a,-\a/2) rectangle (3.5*\a,4.5*\a);
\fill[black!20] (3.5*\a,-\a/2) rectangle (4.5*\a,4.5*\a);
\foreach \x in {0,...,4}
  \foreach \y in {0,...,4}
    \node[circle,inner sep= 1 pt,fill, black!30] at (\x*\a,\a*\y) {};
\node[circle, fill, inner sep=1.5pt] at (0,4*\a) {};
\node[circle, fill, white,inner sep=1.5pt] at (1*\a, 3*\a) {};
\node[circle, fill, inner sep=1.5pt] at (2*\a, 2*\a) {};
\node[circle, fill, inner sep=1.5pt] at (3*\a, \a) {};
\node[circle, fill, inner sep=1.5pt] at (4*\a, 3*\a) {};
\draw[<-] (5*\a, 3*\a) -- (6*\a, 3*\a) node[label=right:$j$] {};
\end{scope}

\end{tikzpicture}
\end{center}
\caption{Schemes for the combinatorics of the chains $\Phi$, $\Omega_j$, $\Theta_i$, and $\Theta_{i,j}$.}
\label{fig:chains}
\end{figure}
Explicitly the signs in these chains are as follows.
If $\sigma$ denotes the facet $\langle(1,1),\dots,(r-1,r-1)\rangle$
of $\Delta_{r,r-1}$,
such that $\pi\sigma=\langle(\pi(1),1),\dots,(\pi(r-1),r-1)\rangle$,
then $\Phi$ is given by
\[
\Phi\ \ =\sum_{\pi_1,\dots,\pi_d\in\Sym_r}
(\textrm{sgn\,}\pi_1)\cdots(\textrm{sgn\,}\pi_d)\,
\pi_1\sigma * \dots * \pi_d\sigma *
\langle (1,1),\dots,(r-1,r-1),(r,r)\rangle
\]
and similarly for $\Omega_j$, $\Theta_i$, and $\Theta_{i,j}$.

The evaluation of $\mathfrak{c}_f$ on $\Phi$ picks out the
facets that correspond to topological colored Tverberg $r$-partitions:
Since the last part of the partition must be the singleton
vertex $v_N$, we find that the last rows of the
chessboard complex $\Delta_{r,r-1}$ factors are not used.
We may define the orientation on $S(W_r^{\oplus(d+1)})$
such that
\[
\mathfrak{c}_f
( \sigma * \cdots * \sigma * \langle (1,1),\dots,(r-1,r-1),(r,r)\rangle)
\ \ =\ \ +\zeta.
\]
Then we get
\begin{multline*}
\mathfrak{c}_f
\big( \pi_1\sigma*\dots*\pi_d\sigma *
\langle (1,1),\dots,(r-1,r-1),(r,r)\rangle\big)
\ =\ \\
\begin{cases}
    (\textrm{sgn\,}\pi_1)\cdots(\textrm{sgn\,}\pi_d)\,\zeta, &
        \textrm{if }\pi_1(r)=\dots=\pi_d(r)=r,\\
    0,& \textrm{otherwise}.
\end{cases}
\end{multline*}
The sign $(\textrm{sgn\,}\pi_1)\cdots(\textrm{sgn\,}\pi_d)$ comes from
the fact that $F$ maps
\[
\sigma*\cdots*\sigma*\langle(1,1),\ldots,(r-1,r-1),(r,r)\rangle
\]
and
\[
\pi_1\sigma*\cdots*\pi_d\sigma*\langle(1,1),\ldots,(r-1,r-1),(r,r)\rangle
\]
to the same simplex in $W_r^{\oplus(d+1)}$, however with a different
order of the vertices. 

Thus,  
\[
\mathfrak{c}_{f}(\Phi ) \ \ =\ \ (r-1)!^{d}\, \zeta .
\]

Furthermore, for any topological colored Tverberg $r$-partition in our configuration the last point $v_N$ has to be a singleton, while the
facets of $\Omega_j$ correspond to colored $r$-partitions where the $j$-th face pairs $v_N$ with a point in $C_d$.
Thus the cochains $\Omega_j$  do not capture any Tverberg partitions, and we get
\begin{equation*}\qquad
\mathfrak{c}_{f}(\Omega_{j})\ \ =\ \ 0\qquad\textrm{ for }1\le j<r.
\end{equation*}

\smallskip

\noindent
\textbf{Is the cocycle $\mathfrak{c}_{f}$ a coboundary?}
Let us assume that $\mathfrak{c}_{f}$ is a coboundary.
Then there is an equivariant cochain
$\mathfrak{h}\in C_{\Sym_r}^{N-1}\big(X,A;\mathcal{Z}\big)$ such that
$\mathfrak{c}_{f}=\delta \mathfrak{h}$, where $\delta $ is
the coboundary operator.

In order to simplify the notation,
from now on we drop the join factor $( \Delta_{r,r-1}) ^{*d}$ from the
notation of the subcomplexes $\Phi $, $\Theta_{i}$ and $\Omega_{i}$.
Note that the join with this complex accounts for a global
sign of $(-1)^{d(r-1)}$ in the boundary/coboundary operators,
since in our vertex ordering the complex $( \Delta_{r,r-1}) ^{*d}$,
whose facets have $d(r-1)$ vertices, comes first.

Thus we have
\[
\partial\Phi\ \ =\ \ (-1)^{d(r-1)} \sum_{i=1}^{r}(-1)^{i-1}\Theta_{i}
\]
and similarly for $1\le j<r$,
\[
\partial\Omega_j\ \ =\ \ (-1)^{d(r-1)}\big( \sum_{i=1}^{r-1}(-1)^{i-1}\Theta_{i,j} \ + \ (-1)^{r-1}\Theta_r\big).
\]

\emph{Claim 1.}
For $1\le i,j<r$, $i\neq j$ we have $\mathfrak{h}( \Theta_{i,j})=0$.

\begin{proof}
    We consider the effect of the transposition $\pi_{ir}:=(ir)=\binom{\ldots i\ldots r}{\ldots r\ldots i}$.
    The simplex
    \[\langle (1,1),\dots,\widehat{(i,i)},\dots,(r-1,r-1),(j,r)\rangle\]
    has no vertex in the $i$-th and in the $r$-th row, so it is fixed
    by $\pi_{ir}$.  The $d$ chessboard complexes
    in  $\Theta_{i,j}$ are invariant but change
     orientation under
    the action of $\pi_{ir}$, so the effect on the chain $\Theta_{i,j}$ is
    $\pi_{ir} \cdot\Theta_{i,j}=(-1)^d  \Theta_{i,j}$
    and hence
    \[
    \mathfrak{h}(\pi_{ir} \cdot\Theta_{i,j})\ =\
    \mathfrak{h}((-1)^d \Theta_{i,j})\ =\
    (-1)^d \mathfrak{h}(\Theta_{i,j}).
    \]
    On the other hand $\mathfrak{h}$ is equivariant, so
    \[
    \mathfrak{h}(\pi_{ir} \cdot\Theta_{i,j})\ =\
    \pi_{ir} \cdot \mathfrak{h}(\Theta_{i,j})\ =\
    (-1)^{d+1} \mathfrak{h}(\Theta_{i,j})
    \]
    since $\Sym_r$ acts on $\mathcal Z$ by multiplication with
    $(\mathrm{sgn\,}\pi)^{d+1}$.

Comparing the two evaluations of
$\mathfrak{h}(\pi_{ir} \cdot\Theta_{i,j})$
yields
$(-1)^{d}   \mathfrak{h}(\Theta_{i,j}) =
 (-1)^{d+1} \mathfrak{h}(\Theta_{i,j}) $.
\end{proof}

\emph{Claim 2.}
For $1\le j<r$ we have
$\mathfrak{h}( \Theta_{j,j})= -\mathfrak{h}(\Theta_j)$.

\begin{proof}
    The interchange of the $j$-th row with the $r$-th
    moves $\Theta_{j,j}$ to $\Theta_j$, where we have to account
    for $d$ orientation changes for the chessboard join factors.

    Thus $\pi_{jr}\Theta_{j,j} = (-1)^d \Theta_j$, which yields
    \[
    (-1)^d\mathfrak{h}( \Theta_j ) \ =\
    \mathfrak{h}((-1)^d \Theta_j ) \ =\
    \mathfrak{h}(\pi_{jr}\Theta_{j,j} ) \ =\
    \pi_{jr}\cdot\mathfrak{h}(\Theta_{j,j} ) \ =\
    (-1)^{d+1}\mathfrak{h}(\Theta_{j,j} ).
    \]
\vskip-10mm
\end{proof}
\medskip

We now use the two claims to evaluate $\mathfrak{h}(\partial\Omega_j)$.
Thus we obtain
\begin{eqnarray*}
        0
    \ =\ \mathfrak{c}_{f}(\Omega_j)
    \ =\ \delta \mathfrak{h}(\Omega_j)
    \ =\ \mathfrak{h}(\partial\Omega_j)
& = &
(-1)^{d(r-1)}\big( (-1)^{j-1}\mathfrak{h}(\Theta_{j,j}) \ + \
                   (-1)^{r-1}\mathfrak{h}(\Theta_r)\big)
\end{eqnarray*}
and hence
\[
(-1)^{j}\mathfrak{h}(\Theta_j)\ \ =\ \ (-1)^{r}\mathfrak{h}(\Theta_r).
\]

The final blow now comes from our earlier evaluation
of the cochain $\mathfrak{c}_f$ on $\Phi$:
\begin{eqnarray*}
    (r-1)!^{d} \cdot \zeta
\ =\ \mathfrak{c}_{f}(\Phi)
\ =\ \delta \mathfrak{h}(\Phi)
\ =\ \mathfrak{h}(\partial \Phi )
& =& \mathfrak{h}( (-1)^{d(r-1)} \sum_{j=1}^{r} (-1)^{j-1} \Theta_j)\\
& =& -(-1)^{d(r-1)}\sum_{j=1}^{r}(-1)^{j}\mathfrak{h}( \Theta_j)\\
& =& -(-1)^{d(r-1)}\sum_{j=1}^{r}(-1)^{r}\mathfrak{h}( \Theta_r)\\
& =& (-1)^{(d+1)(r-1)} r\,\mathfrak{h}( \Theta_r).
\end{eqnarray*}

Thus, the integer coefficient of $\mathfrak{h}(\Theta _{r})$
should be equal to $\tfrac{(r-1)!^{d}}{r}\zeta$, up to a sign. Consequently, 
when $r~\nmid ~(r-1)!^{d}$, the cocycle $\mathfrak{c}_{f}$ is not a coboundary, 
i.e., the cohomology class $\mathfrak{o}=[\mathfrak{c}_{f}]$
does not vanish and so there is no $\mathfrak{S}_{r}$-equivariant extension 
$X\rightarrow S(W_{r}^{\oplus (d+1)})$ of $F_0|_A$. 

On the other hand, when $r\mid(r-1)!^{d}$ we can define
\begin{equation}
\label{eqDefinitionOfH}
\begin{array}{llll}
\mathfrak{h}(\Theta _{j}) & := & + (-1)^{\left( d+1\right) \left(
r-1\right) +j+r}\cdot\tfrac{(r-1)!^{d}}{r}\cdot \zeta , & \text{for
}1\leq j\leq
r, \\
\mathfrak{h}(\Theta _{j,j}) & := & - (-1)^{\left( d+1\right) \left(
r-1\right) +j+r}\cdot\tfrac{(r-1)!^{d}}{r}\cdot \zeta , & \text{for
}1\leq j<r,
\\[1.6mm]
\mathfrak{h}(\Theta _{i,j}) & := & 0, & \text{for }i\neq j,~1\leq
i\leq
r,~1\leq j<r.
\end{array}
\end{equation}
Here we do obstruction theory with respect to the filtration
$(\Delta_{r,r-1})^{*d} * (\Delta_{r,r-1}*[r])^{(n)}$ of~$X$, where
$(\Delta_{r,r-1}*[r])^{(n)}$ denotes the $n$-skeleton of
$\Delta_{r,r-1}*[r]$.
The ``cells'' are of the form $(\Delta_{r,r-1})^{*d}*F$, where $F$ ranges over the faces of $\Delta_{r,r-1}*[r]$.
They are connected oriented pseudomanifolds with boundary, their boundary being the pseudomanifolds $(\Delta_{r,r-1})^{*d}*\partial F$.
If $\dim (\Delta_{r,r-1})^{*d}*\partial F=N-1=\dim S(W_r^{\oplus(d+1)})$, then a map $(\Delta_{r,r-1})^{*d}*\partial F\to S(W_r^{\oplus(d+1)})$ can be \emph{non-equivariantly} extended to a map
 $(\Delta_{r,r-1})^{*d}*F\to S(W_r^{\oplus(d+1)})$ if and only if its degree is zero; this uses a standard obstruction theory argument. 
Similarly, if $\dim (\Delta_{r,r-1})^{*d}* F=N-1=\dim S(W_r^{\oplus(d+1)})$,
then the set of non-equivariant extensions of a map $(\Delta_{r,r-1})^{*d}*\partial F\to S(W_r^{\oplus(d+1)})$ to $(\Delta_{r,r-1})^{*d}*F\to S(W_r^{\oplus(d+1)})$ corresponds bijectively to the elements in
the group
$H^N((\Delta_{r,r-1})^{*d}*(F,\partial F);\mathbb{Z})=\mathbb{Z}$; the bijection depends on a choice of one extension that should correspond to $0\in\mathbb{Z}$.
The obstruction cocycle $\mathfrak{c}_{f}$ can thus be regarded as an element in the simplicial cochain complex
\[
C^{r-1}_{\mathfrak{S}_{r}} (\Delta_{r,r-1}*[r], B ; \mathcal{Z} \otimes
H_{(r-1)d-1} ((\Delta_{r,r-1})^{*d};\mathbb{Z})),
\]
where $B$ denotes the subcomplex of $\Delta_{r,r-1}*[r]$ on which $\Sym_r$ does not act freely.
The coefficients are twisted with the top homology of $(\Delta_{r,r-1})^{*d}$ in order to account for the $\mathfrak{S}_{r}$-action on the orientation of the cells.
The coboundary of $\mathfrak{h}$ as defined in~\eqref{eqDefinitionOfH} is $\mathfrak{c}_{f}$.
Since
$\mathfrak{h}$ is only non-zero on the cells $\Theta_{j}$ and
$\Theta_{j,j}$, which are only invariant under
$\textnormal{id}\in\mathfrak{S}_{r}$,  
we can solve the extension problem equivariantly.
Note also that this map still coincides with~$F_0$ on $A$.

Hence for $r\mid(r-1)!^{d}$  an 
$\mathfrak{S}_{r}$-equivariant extension $X\rightarrow S(W_{r}^{\oplus (d+1)})$ of $F_0|_A$ exists.
\end{proof}

\begin{remark} (February, 2013)\\
We are happy that our work has attracted a lot of attention immediately after
the first presentation (at IPAM, Los Angeles) in October 2009.

Soon after completion of the first version of the preprint for this paper we noticed (see \cite[Sect.~2]{BMZ2})
that the non-existence part of Proposition~\ref{prop:main} can also be phrased in more elementary terms
using degrees rather than by using equivariant obstruction theory;
this was also noticed by Vre\'cica and \v{Z}ivaljevi\'c~\cite{VZdegreeproof}.

We note that despite the condition $r\mid(r-1)!^d$ obtained from evaluation of the obstruction cocycle
on a particular subcomplex, the correct value for the degree of the equivariant map in question is $(r-1)!^{d+1}$,
such that the degree approach only yields the necessary condition 
$r\mid(r-1)!^{d+1}$ for the existence of the map.

We provide the degree formulation of the proof of non-existence part of Proposition~\ref{prop:main} in \cite{BMZ2}
as a special case of a Tverberg--Vre\'cica type transversal theorem,
accompanied by much more complete cohomological index calculations,
which also yield a second new proof that establishes Theorem \ref{thm:main2}
directly, without a reduction to Theorem~\ref{thm:main}.
Matou\v{s}ek, Tancer \& Wagner \cite{MTW} have presented an geometric version of the 
degree-based proof for the non-existence part of Proposition~\ref{prop:main}.

The proof in terms of degrees, however, does not imply that the $\mathfrak{S}_{r}$-equivariant map
proposed by the natural configuration space/test map scheme
of Theorem~\ref{prop:main} exists if $r$ divides $(r-1)!^d$. 
Moreover, the non-existence of an induced equivariant map in the case $d=1$ and $r=4$ can only be captured by the use of equivariant obstruction theory.

See \cite{Z123e} for an exposition of the history and context of this subject.
\end{remark}

\bigskip
\footnotesize
\noindent\textit{Acknowledgements.}
We are grateful to Carsten Schultz for
critical comments, to Marie-Sophie Litz for help with the images, to an \emph{Annals} referee for an extremely
careful and helpful report, and to Aleksandra, Julia, and Torsten for constant
support. The Mathematische Forschungsinstitut Oberwolfach and The
Institute for Pure and Applied Mathematics at UCLA provided perfect
working environments for the completion of this paper. 

The research leading to these results has received funding from the European Research Council under the European Union's 
Seventh Framework Programme (FP7/2007-2013)/ERC Grant agreement no.\ 247029-SDModels.
The first author was also supported by the grant ON 174008 of the Serbian Ministry of Education and Science.
The second author was supported by Deutsche Telekom Stiftung, IAS, EPDI, and Max Planck Institute for Mathematics.



%
\normalsize


\end{document}